\documentclass[12pt,leqno,a4paper]{article}
\usepackage{amsmath,amsthm,enumerate}
\usepackage[latin1]{inputenc}
\usepackage[T1]{fontenc}
\usepackage[english]{babel}

\newtheorem{theorem}{Theorem}

\newtheorem{definition}[theorem]{Definition}



\newcommand*\proofnamestyle{\itshape}
 \makeatother

\DeclareMathOperator{\tr}{Tr}

\begin{document}

    \title{Some operator monotone functions}
      \author{Frank Hansen}
      \date{March 16, 2008\\
      \tiny Final version}

      \maketitle

      \begin{abstract}
     We prove that the functions $ t\to (t^q-1)(t^p-1)^{-1} $ are operator monotone in the positive
     half-axis for $ 0<p\le q\le 1, $ and we calculate the two associated canonical representation formulae.
     The result is used to find new monotone metrics (quantum Fisher information) on the state space of quantum systems.
      \end{abstract}

      \vskip 2ex

      {\footnotesize\noindent
      Dedicated to Professor Jun Tomiyama on his 77th birthday with respect and affection}

       \section{Introduction}

    We consider the functions
    \begin{equation}\label{definition of f}
    f(t)=\left\{\begin{array}{ll}
    \displaystyle\frac{p}{q}\cdot\frac{t^q-1}{t^p-1} \qquad &t>0,\: t\ne 1\\[3ex]
    \displaystyle 1 &t=1
    \end{array}\right.
    \end{equation}
    for positive exponents $ p $ and $ q. $

    \begin{theorem}\label{main theorem}
    The function $ f $ in (\ref{definition of f}) is operator monotone for $ 0<p\le q\le 1. $
    \end{theorem}

    \begin{proof} By an elementary calculation we may write
    \[
    f(t)=\int_0^1 (\lambda t^p + (1-\lambda))^{(q-p)/p}\, d\lambda.
    \]
    For $ \Im z>0 $ and $ 0<\lambda <1 $ we have $ 0<\arg( \lambda z^p +1-\lambda) < p\pi. $ Thus
    \[
    0<\arg( \lambda z^p +1-\lambda)^{(q-p)/p} < (q-p)\pi\le \pi.
    \]
    This shows that the integrand function and hence $ f $ is operator monotone \cite{kn:loewner:1934, kn:donoghue:1974, kn:bhatia:1997}.
    \end{proof}

    Let $ z=re^{i\theta} $ with $ r>0 $ and $ 0<\theta<\pi. $ Then the imaginary part
    \[
    \Im f(z)=\frac{p}{q}\cdot\frac{r^{p+q}\sin(q-p)\theta - r^q\sin q\theta + r^p\sin p\theta}{r^{2p}-2r^p\cos p\theta+1}
    \]
    and since $ f $ is operator monotone and non-constant, we thus have
    \begin{equation}\label{nominator of imaginary part}
    r^{p+q}\sin(q-p)\theta - r^q\sin q\theta + r^p\sin p\theta> 0
    \end{equation}
    for $ r>0 $ and $ 0<\theta<\pi. $ We need this result later in the paper.

    In the first version of the paper we proved (\ref{nominator of imaginary part}) directly to obtain operator monotonicity of $ f. $
    Then Furuta gave an elementary proof using the techniques developed in \cite[Proposition 3.1]{kn:furuta:2008}. Finally, Ando gave the above
    proof which is the shortest known to the author.

    \section{Integral representations}

    \begin{theorem}
    The function $ f $ in (\ref{definition of f}) has the canonical representation
    \[
    f(t)=\frac{p}{q}+\frac{p}{q}\int_0^\infty\frac{t}{\lambda (t+\lambda)}\cdot
     \frac{\lambda^{p+q}\sin(q-p)\pi - \lambda^q\sin q\pi
    + \lambda^p\sin p\pi}{\pi(\lambda^{2p}-2\lambda^p\cos p\pi+1)}\,d\lambda
    \]
    for $ 0<p\le q\le 1. $
    \end{theorem}

    \begin{proof} The representing measure of the operator monotone function $ f $ is calculated by first
    considering the analytic extension $ f(z) $ to the upper complex half-plane, cf. \cite{kn:donoghue:1974}.
    If $ z=r e^{i\theta} $ approaches a real $ \lambda<0 $ from the upper complex half-plane, then
    $ r\to -\lambda $ and $ \theta\to\pi. $ Consequently, the imaginary part
    \[
    \Im f(z)\to \frac{p}{q}\cdot\frac{(-\lambda)^{p+q}\sin(q-p)\pi - (-\lambda)^q\sin q\pi
    + (-\lambda)^p\sin p\pi}{(-\lambda)^{2p}-2(-\lambda)^p\cos p\pi+1}\,.
    \]
    If $ z=r e^{i\theta} $ approaches zero from the upper complex half-plane, then $ \theta $ is
    indeterminate but $ r\to 0 $ and $ \Im f(z)\to 0. $ The representing measure \cite[Chapter II Lemma 1]{kn:donoghue:1974} is thus given by
     \[
    d\mu(\lambda)=\frac{p}{q}\cdot
    \frac{\lambda^{p+q}\sin(q-p)\pi - \lambda^q\sin q\pi
    + \lambda^p\sin p\pi}{\pi(\lambda^{2p}-2\lambda^p\cos p\pi+1)}\,d\lambda.
    \]
    Since $ f $ is both positive and operator monotone it is necessarily of the form
    \[
    f(t)=\alpha t +f(0) + \int_0^\infty\frac{t}{\lambda (t+\lambda)}\,d\mu(\lambda)
    \]
    where $ \alpha\ge 0 $ and the representing measure $ \mu $ satisfies
    \[
    \int_0^\infty (\lambda^2+1)^{-1}\,d\mu(\lambda)<\infty\quad\mbox{and}\quad
    \int_0^1 \lambda^{-1}\,d\mu(\lambda)<\infty.
    \]
    Finally, since the growth of $ f(t) $ is smaller than the growth of $ t $ in infinity we obtain $ \alpha=0, $
    and the statement follows.
    \end{proof}

    \begin{theorem}
    Let $ 0<p\le q\le 1. $ The function $ f $ in (\ref{definition of f}) has the canonical exponential representation
    \[
    f(t)=\frac{p}{q}\left(\frac{1-\cos q\frac{\pi}{2}}{1-\cos p\frac{\pi}{2}}\right)^{1/2}
    \exp\int_0^\infty
    \left(\frac{\lambda}{\lambda^2+1}-\frac{1}{\lambda+t}\right)h(\lambda)\,d\lambda.
    \]
    where
    \begin{equation}\label{generating weight function}
    h(\lambda)=\frac{1}{\pi}\arctan\frac{\lambda^{p+q}\sin(q-p)\pi - \lambda^q\sin q\pi +\lambda^p\sin p\pi}
    {\lambda^{p+q}\cos(q-p)\pi - \lambda^q\cos q\pi - \lambda^p\cos p\pi +1}\quad \lambda>0.
    \end{equation}
    (Notice that $ 0\le h(\lambda)\le 1/2 $ for every $ \lambda>0). $
    \end{theorem}

    \begin{proof}
    The exponential representation of $ f $ is obtained by considering the operator monotone function
    $ \log f(t). $ The analytic continuation $ \log f(z) $ to the upper complex half-plane
    has positive imaginary part bounded by $ \pi, $ cf. \cite{kn:aronszajn:1956, kn:donoghue:1974, kn:hansen:1981}.
    The representing measure of $ \log f(t) $ is therefore
    absolutely continuous with respect to Lebesgue measure with Radon-Nikodym derivative bounded by one. We calculate
    and obtain the expression
    \[
    \begin{array}{rl}
    \Im\log f(z)&\displaystyle
    =\frac{1}{2i}\log\frac{(r^q e^{iq\theta}-1)(r^pe^{-ip\theta}-1)}{(r^pe^{ip\theta}-1)(r^q e^{-iq\theta}-1)}\\[3ex]
    &\displaystyle
    =\frac{1}{2i}\log\left[\frac{a+ib}
    {(r^{2p}-2r^p\cos p\theta+1)^{1/2}(r^{2q}-2r^q\cos q\theta+1)^{1/2}}\right]^2,\\[3ex]
    \end{array}
    \]
    where
    \[
    \begin{array}{rl}
    a&=r^{p+q}\cos(q-p)\theta - r^q\cos q\theta - r^p\cos p\theta +1,\\[1ex]
    b&=r^{p+q}\sin(q-p)\theta - r^q\sin q\theta +r^p\sin p\theta.
    \end{array}
    \]
    We recognize from (\ref{nominator of imaginary part}) that $ b>0, $ and since the imaginary part of $ (a+ib)^2 $ is positive, we obtain
    that also $ a>0. $
    The expression inside the square is thus a complex number of modulus one in the first
    quadrant of the complex plane. It is of the form $ \exp i\phi(z) $ for some $ \phi(z) $ with
    $ 0<\phi(z)<\pi/2, $ where
    \[
    \phi(z)=\arctan\frac{r^{p+q}\sin(q-p)\theta - r^q\sin q\theta +r^p\sin p\theta}
    {r^{p+q}\cos(q-p)\theta - r^q\cos q\theta - r^p\cos p\theta +1}
    \]
    and $ \Im\log f(z)=\phi(z). $
    We let $ z=r e^{i\theta} $ approach a real $ \lambda<0 $ from the upper complex half-plane and obtain
    \[
    \Im\log f(z)\to\arctan\frac{(-\lambda)^{p+q}\sin(q-p)\pi - (-\lambda)^q\sin q\pi +(-\lambda)^p\sin p\pi}
    {(-\lambda)^{p+q}\cos(q-p)\pi - (-\lambda)^q\cos q\pi - (-\lambda)^p\cos p\pi +1}\,.
    \]
    The representing measure is thus given by the weight function
    \[
    h(\lambda)=\frac{1}{\pi}\arctan\frac{\lambda^{p+q}\sin(q-p)\pi - \lambda^q\sin q\pi +\lambda^p\sin p\pi}
    {\lambda^{p+q}\cos(q-p)\pi - \lambda^q\cos q\pi - \lambda^p\cos p\pi +1}\quad \lambda>0,
    \]
    and we obtain the exponential representation
    \[
    f(t)=\exp\left[\beta+\int_0^\infty
    \left(\frac{\lambda}{\lambda^2+1}-\frac{1}{\lambda+t}\right)h(\lambda)\,d\lambda\right],
    \]
    where $ \beta=\Re\log f(i). $ By a tedious calculation we obtain
    \[
    \beta=\log p-\log q+\frac{1}{2}\log\frac{1-\cos q\frac{\pi}{2}}{1-\cos p\frac{\pi}{2}}
    \]
    and the statement is proved.
    \end{proof}

    \section{Applications to quantum information theory}

    \begin{definition}
    We denote by $ \mathcal F_{\text{op}} $ the set of functions $ f\colon\mathbf R_+\to\mathbf R_+ $ satisfying
    \begin{enumerate}[(i)]
    \item $ f $ is operator monotone,
    \item $ f(t)=t f(t^{-1}) $ for all $ t>0, $
    \item $ f(1)=1. $
    \end{enumerate}
    \end{definition}

    The following result was proved in \cite[Theorem 2.1]{kn:audenaert:2008}.

    \begin{theorem}\label{characterization of F_op}
    A function $ f\in\mathcal F_{\text{op}} $ admits a canonical representation
    \begin{gather}\label{canonical representation of f}
    f(t)=\frac{1+t}{2}\exp\int_0^1\frac{(\lambda^2-1)(1-t)^2}{(\lambda+t)(1+\lambda t)(1+\lambda)^2}\,h(\lambda)\,d\lambda,
    \end{gather}
    where the weight function $ h:[0,1]\to[0,1] $ is measurable.
    The equivalence class containing $ h $ is uniquely determined by $ f. $ Any function
    on the given form is in $ \mathcal F_{\text{op}}. $
    \end{theorem}

    In addition, for $ z=r e^{i\theta} $ with $ r>0 $ and $ 0<\theta<\pi, $ the weight function $ h $ in the above theorem appears as
    \[
    h(\lambda)=\frac{1}{\pi} \lim_{z\to -\lambda}\Im\log f(z)
    \]
    for almost all $ \lambda\in(0,1]. $ Notice that $ r\to \lambda $ and $ \theta\to\pi $ when $ z\to -\lambda. $

    A monotone metric is a map $ \rho\to K_\rho(A,B) $ from the set $ \mathcal M_n $ of positive definite $ n\times n $ density matrices
    to sesquilinear forms $ K_\rho(A,B) $ defined on $ M_n(\mathbf C) $ satisfying:

    \begin{enumerate}

    \item $ K_\rho(A,A)\ge 0, $ and equality holds if and only if $ A=0. $

    \item $ K_\rho(A,B)=K_\rho(B^*,A^*) $ for all $ \rho\in\mathcal M_n $ and all $ A,B\in M_n(\mathbf C). $

    \item $ \rho\to K_\rho(A,A) $ is continuous on $ \mathcal M_n $ for every $ A\in M_n(\mathbf C). $

    \item $ K_{T(\rho)}(T(A),T(A))\le K_\rho(A,A) $ for every $ \rho\in\mathcal M_n, $ every
    $ A\in M_n(\mathbf C) $ and every stochastic mapping $ T:M_n(\mathbf C)\to M_m(\mathbf C). $

    \end{enumerate}
    A mapping $ T:M_n(\mathbf C)\to M_m(\mathbf C) $ is said to be stochastic if it is completely positive and trace preserving.
    A monotone metric \cite{kn:censov:1982, kn:petz:1996:2} is given on the form
    \begin{equation}\label{Morozova-Chentsov function}
    K_\rho(A,B)=\tr A^* c(L_\rho,R_\rho) B,
    \end{equation}
    where $ c $ is a so called
    Morozova-Chentsov function and $ c(L_\rho,R_\rho) $ is the function taken
    in the pair of commuting left and right multiplication operators (denoted
    $ L_\rho $ and $ R_\rho, $ respectively) by $ \rho. $ The Morozova-Chentsov functions are of the form
    \begin{equation}\label{Morozova-Chentsov function}
    c(x,y)=\frac{1}{y g(x y^{-1})}\qquad x,y >0,
    \end{equation}
    where $ g\in \mathcal F_{\text{op}}.  $

     There is an involution $ f\to f^{\#} $
    on the set of positive operator monotone functions $ f $ defined in the positive half-axis given by
    \[
    f^{\#}(t)=t f(t^{-1})\qquad t>0,
    \]
    cf. \cite{kn:hansen:2006:2}. It plays a role in the following result.

    \begin{theorem}
    The function
    \[
    c(x,y)=\frac{q}{p}\cdot\frac{x^p-y^p}{x^q-y^q}\, (xy)^{-(1-q+p)/2}\qquad x,y>0
    \]
    is a Morozova-Chentsov function for $ 0<p\le q\le 1. $ The generating operator monotone function $ g $ in (\ref{Morozova-Chentsov function}) has the exponential representation
    \[
    g(t)=\frac{1+t}{2}\exp\int_0^1\frac{(\lambda^2-1)(1-t)^2}{(\lambda+t)(1+\lambda t)(1+\lambda)^2}\,\left(\frac{1-q+p}{2} + h(\lambda)\right) d\lambda,
    \]
    where the weight function $ h $ is the restriction to the unit interval of the function given in (\ref{generating weight function}).
    \end{theorem}

    \begin{proof}
    We first notice that the function
    \begin{equation}\label{generating operator monotone function}
    g(t)=\frac{p}{q}\cdot \frac{t^q -1}{t^p -1}\,t^{(1-q+p)/2}\qquad t>0
    \end{equation}
    generates $ c(x,y) $ according to (\ref{Morozova-Chentsov function}). We thus have to prove that $ g\in \mathcal F_{\text{op}} $ for $ 0<p\le q\le 1. $
    Choosing the function $ f $ as defined in (\ref{definition of f}), we obtain
    \[
    f^{\#}(t)=t^{1-q+p} f(t)\qquad t>0.
    \]
    Since obviously the function
    \[
    \sqrt{f(t) f^{\#}(t)}=t^{(1-q+p)/2} f(t)\qquad t>0
    \]
    is a fix-point under the involution $ \# $ and the geometric mean is operator monotone, we obtain $ g\in \mathcal F_{\text{op}} $ as desired. We have thus proved the first part of the theorem.
    Setting $ z=r e^{i\theta} $ for $ r>0 $ and $ 0<\theta<\pi $ we calculate that
    \[
    \Im\log g(z)=\frac{1-q+p}{2}\theta + \Im\log f(z).
    \]
    The result now follows by Theorem~\ref{characterization of F_op} and the remarks below.
    \end{proof}

    \nocite{kn:bhatia:1997}

 {\footnotesize


      \vfill

      \noindent Frank Hansen: Department of Economics, University
       of Copenhagen, Studiestraede 6, DK-1455 Copenhagen K, Denmark.}

      \end{document}